\documentclass[a4paper]{article}
\usepackage{amsfonts}
\usepackage{amssymb,amsthm, amsmath,graphicx,bm,ebezier,amsthm}
\usepackage{hyperref}
\usepackage{caption}
\usepackage{url,lineno}
\usepackage{arydshln}
\usepackage[noend]{algpseudocode}
\usepackage[ruled,vlined,linesnumbered]{algorithm2e}
\usepackage{lscape}
\usepackage{multirow}
\usepackage[dvipsnames]{xcolor}

\usepackage{tikz}
\usetikzlibrary{arrows.meta}
\tikzset{
  my tree/.style={
    <-, 
    nodes={minimum size = .25cm},
    >=Stealth[],
   level 1/.style={sibling distance=30pt}
  },
}

\newtheorem{theorem}{Theorem}

\newtheorem{proposition}[theorem]{Proposition}
\newtheorem{corollary}[theorem]{Corollary}
\newtheorem{lemma}[theorem]{Lemma}

\theoremstyle{remark}

\newtheorem{example}[theorem]{Example}
\newtheorem{remark}[theorem]{Remark}

\usepackage{todonotes}

\def\CaH{\mathcal{H}}
\def\CaD{\mathcal{D}}
\def\CaC{\mathcal{C}}
\def\CaA{\mathcal{A}}
\def\CaB{\mathcal{B}}

\def\N{\mathbb{N}}
\def\Z{\mathbb{Z}}
\def\Q{\mathbb{Q}}
\def\R{\mathbb{R}}
\def\Fb{\mathrm{Fb}}

\def\k{\mathrm{k}}

\title{On the quotient of affine semigroups by a positive integer}

\date{}
\author{
J. I. Garc\'{\i}a-Garc\'{\i}a,
R. Tapia-Ramos,
and A. Vigneron-Tenorio
}

\begin{document}


\maketitle

\begin{abstract}

This work delves into the {\it quotient of an affine semigroup by a positive integer}, exploring its intricate properties and broader implications. We unveil an {\it associated tree} that serves as a valuable tool for further analysis. Moreover, we successfully generalize several key irreducibility results, extending their applicability to the more general class of $\CaC$-semigroup quotients. To shed light on these concepts, we introduce the novel notion of an {\it arithmetic variety of affine semigroups}, accompanied by illuminating examples that showcase its power.

\end{abstract}

{\small


}

\section*{Introduction}

We consider an {\em affine semigroup} $S$ as a non-empty subset of $\N^p$ ($\N$ denoting the set of non-negative integers) containing the zero element, closed under the usual addition in $\N^p$, and such that there exists a finite subset $A=\{a_1,\ldots ,a_q\}$ of $S$ with $S=\{\sum_{i=1}^q\lambda_ia_i\mid \lambda_1,\ldots ,\lambda_q\in\N\}$. A fundamental result states that each affine semigroup possesses a unique minimal generating set (see \cite{libroRosalesMonoids}). When  $p=1$, and the elements of $A$ are coprime, we say that $S$ is a {\em numerical semigroup}. If $p> 1$, $\CaC$ the minimal positive integer cone containing to $S$, and $S$ satisfies that $\CaC \setminus S$ is finite, $S$ is called {\em $\CaC$-semigroup}  (see \cite{Csemigroup}). Notably, when $\CaC=\N^p$, the semigroup $S$ is known as {\em generalized numerical semigroup} (see \cite{GenSemNp}).

Given an affine semigroup $S \subseteq \N^p$, its quotient by a positive integer $d$ is defined as $\frac S d = \{x \in \N^p \mid dx \in S\}$. This construction plays a significant role in the study of numerical semigroups, appearing in numerous works such as \cite{bogart}, \cite{OR2020}, \cite{JCR-PAGS-2008}, and \cite{Swanson}. However, despite its relevance in numerical contexts, a systematic investigation of quotients of affine semigroups, in general, is surprisingly absent in the literature. We address this gap by presenting various results about quotients of affine semigroups by a positive integer. While some of these results are known for numerical semigroups, others hold true only for general affine semigroups with $p>1$ or $\CaC$-semigroups.

We present an efficient algorithm for computing a generating set of the quotient of any affine semigroup by a positive integer. This algorithm provides a valuable tool for further analysis of such quotients. Moreover, we demonstrate that the families of convex body semigroups and Cohen-Macaulay simplicial semigroups exhibit the remarkable property of being closed under quotients. This closure property unveils a powerful tool for generating infinite Cohen-Macaulay semigroups from a single instance, paving the way for exploring their properties and applications.

A central focus of this work is identifying all $\CaC$-semigroups that are quotients of a given $\CaC$-semigroup (Theorem \ref{theorem_D_d}). Inspired by this set, we construct a directed graph under specific hypotheses. Additionally, we delve into arithmetic varieties (non-empty sets closed under intersection and quotient by any positive integer), introducing three distinct families of affine semigroups that form three distinct arithmetic varieties.

Another key objective is to study the irreducibility of quotients of $\CaC$-semigroups. As shown in \cite{SomepropCsemgp}, every irreducible semigroup is either symmetric or pseudo-symmetric. We demonstrate that any $\CaC$-semigroup is part of infinitely many symmetric $\CaC$-semigroups and infinitely many pseudo-symmetric $\CaC$-semigroups. Furthermore, we provide a characterization of irreducible $\CaC$-semigroups relative to the quotient of $\CaC$-semigroups (Theorem \ref{cairrquo}).

Our paper is structured as follows. Section~\ref{pre} establishes the necessary background for subsequent sections. In Section~\ref{thequotient}, we introduce the concept of quotients of affine semigroups, along with their corresponding systems of generators. Section~\ref{compD_dTree} delves into the set of all $\CaC$-semigroups $T$ for which $S=\frac{T}{d}$. This section also introduces the novel concept of arithmetic varieties of affine semigroups, employed to compute an associated tree for these semigroups. Section~\ref{irre} presents several new results on the irreducibility of quotients of $\CaC$-semigroups, drawing inspiration from Chapter 5, Section 2 of \cite{libroRosales}. Finally, Section~\ref{av} revisits the exploration of arithmetic varieties, showcasing their application through illustrative examples.

\section{Preliminars}
\label{pre}

Let $\Q$ and $\R_\geq$ denote the sets of rational and non-negative real numbers, respectively. For any $d \in \N$, define $[d] = \{0, \ldots, d\}$. Given $X \subseteq \N^p$, let $dX = \{dx \mid x \in X\}$.

This work employs various orders on different sets. We denote by $\leq$ the usual partial order on $\N^p$, where $x \leq y$ implies $y - x \in \N^p$. Additionally, we consider a total order $\preceq$ on $\N^p$ satisfying compatibility with addition and ensuring $0 \preceq c$ for any $c \in \N^p$. Throughout this work, we fix a specific total order $\preceq$.

For a $\CaC$-semigroup $S$, any element $x \in \CaC \setminus S$ is called a {\em gap} of $S$. We call $x \in S$ a \emph{pseudo-Frobenius element} if $x + (S \setminus \{0\}) \subset S$. The set of gaps of $S$ is denoted by $\CaH(S)$, and its set of pseudo-Frobenius elements by $PF(S)$. The cardinality of $\CaH(S)$ is known as the {\em genus} of $S$. Following these definitions, the {\em Frobenius element} of $S$, denoted by $\Fb(S)$, is the maximum element in $\CaH(S)$ with respect to the fixed total order on $\N^p$. Note that $\Fb(S) = \max_\preceq PF(S)$.

The following proposition provides a useful characterization of the gaps of a $\CaC$-semigroup in terms of its pseudo-Frobenius elements.

\begin{proposition}\cite[Corollary 2.15]{AffineSmgp}
\label{carcH(S)}
Let $S$ be a $\CaC$-semigroup and $x\in \CaC$. Then, $x\notin S$ if and only if there exists $f\in PF(S)$ such that $f-x\in S.$
\end{proposition}

A $\CaC$-semigroup $S$ is irreducible if it cannot be expressed as an intersection of two $\CaC$-semigroups containing $S$ properly. Equivalently, $S$ is irreducible if and only if the set $PF(S)$ is equal to $\{\Fb(S)\}$ or $\{\Fb(S),\frac{\Fb(S)}{2}\}$ (see \cite{PsFbCsemigp}). If $PF(S)=\{\Fb(S)\}$, we say that $S$ is symmetric, and pseudo-symmetric when $PF(S)=\{\Fb(S),\frac{\Fb(S)}{2}\}$.

From the relationship between the elements belonging to $S$ and its Frobenius element, the irreducibility of $S$ is determined as follows.

\begin{proposition}\cite[Theorem 3.6]{AffineSmgp}
\label{caracSIM}
    Let $S$ be a $\CaC$-semigroup and $x\in\CaC$. The following conditions are equivalent:
    \begin{itemize}
        \item [(i)] $S$ is a symmetric $\CaC$-semigroup.
        \item [(ii)] $x\in S$ if and only if $\Fb(S)-x\notin S$.  
    \end{itemize}
\end{proposition}

\begin{proposition}\cite[Theorem 3.7]{AffineSmgp}
    \label{caracPSEUDOSIM}
    Let $S$ be a $\CaC$-semigroup and $x\in\CaC$. Then,
    $S$ is a pseudo-symmetric $\CaC$-semigroup if and only if the following conditions hold:
    \begin{itemize}
        \item [(i)] $\Fb(S)\in2\N^p$.
        \item [(ii)] $x\in S$ if and only if $\Fb(S)-x\notin S$ and $x\ne\frac{\Fb(S)}{2}$.  
    \end{itemize}
\end{proposition}

It is well-known that the fundamental gaps of a $\CaC$-semigroup determine it uniquely (see \cite{SomepropCsemgp}). According to the usual literature, a {\em fundamental gap} of a $\CaC$-semigroup $S$ is a gap $x$ of $S$ such that $2x,3x\in S$.  We show that fundamental gaps of the quotient of a $\CaC$-semigroup by a positive integer can be obtained from the fundamental gaps of the original $\CaC$-semigroup.

\begin{proposition}\label{propFG}
Let $S$ be a  $\CaC$-semigroup and $d\in \N\setminus\{0\}$. Then,
\[
FG\left(\frac{S}{d}\right) = \left\{ \frac{h}{d} \mid h\in FG(S) \text{ and } h \equiv 0\mod d  \right\}.
\]
\end{proposition}
\begin{proof}
    Since $d\in \N\setminus\{0\}$, $x$ belongs to $FG\left(\frac{S}{d}\right)$ if and only if $d x\notin S$, and $k d  x \in S$ for any $k\in \N$. Hence, $d x\in FG(S)$.
\end{proof}

The Ap\'ery set of a $\CaC$-semigroup $S$ relative to $m\in S\setminus\{0\}$ is defined as $Ap(S,m)=\{x\in S\mid x-m\in\CaH(S)\}$. If $d\in \N\setminus \{0\}$ divides to $m$, then we can describe $Ap\left(\frac{S}{d},\frac{m}{d}\right)$ in terms of $Ap(S,m)$ as follows.

\begin{proposition}\label{propApery}
    Let $S$ be a $\CaC$-semigroup and $m$ a non-zero element of $S$ such that $d$ divides to $m$. Then
    \[
    Ap\left(\frac{S}{d},\frac{m}{d}\right)=\left\{\frac{w}{d}\mid w\in Ap(S,m)\text{ and } w\equiv 0 \mod d \right\}.
    \]
\end{proposition}
\begin{proof}
    By definition $x\in Ap\left(\frac{S}{d},\frac{m}{d}\right)$ if and only if $x=\frac{w}{d}\in \frac{S}{d}$ and $\frac{w}{d}- \frac{m}{d}\notin\frac{S}{d}$ where $w=d x,$ this means, $w\in Ap(S,m)$ and $w\equiv 0 \mod d$.
\end{proof}

The above propositions \ref{propFG} and \ref{propApery}, generalize to $\CaC$-semigroups Proposition \cite[Proposition 6.2]{libroRosales} and \cite[Proposition 6.5]{libroRosales}, respectively.

\section{Some families of affine semigroups closed under the quotient by a positive integer}
\label{thequotient}

Recall that the quotient of $S$ by $d$ is the set $\frac{S}{d}=\{ s\in \N^p\mid dx\in S \}$.
Consequently, we say that $\frac{S}{2}$ is one-half of $S$, and that $\frac{S}{4}$ is one-fourth of $S$. These specific cases will be relevant later in our work.

This section introduces certain families of affine semigroups where the quotient of any element by a positive integer remains within the same family. Specifically, we focus on convex body semigroups and Cohen-Macaulay simplicial semigroups. Additionally, we present an algorithm for computing the minimal generating set of the quotient of any affine semigroup by a positive integer.

We begin by outlining the algorithm. Let $A = \{a_1, \ldots, a_q\}$ be the minimal generating set of an affine semigroup $S \subseteq \N^p$, and $d \in \N \setminus \{0\}$. By definition of $\frac{S}{d}$, an element $x = (x_1, \ldots, x_p)$ belongs to $\frac{S}{d}$ if and only if there exists $\lambda = (\lambda_1, \ldots, \lambda_q) \in \N^q$ such that $dx = \sum_{i=1}^q \lambda_i a_i$. Equivalently, $x \in \frac{S}{d}$ if and only if $(\lambda, x) \in \N^{q} \times \N^p$ is a non-negative integer solution of the system of linear equations $(M \mid D) (\lambda, x)^t = 0$, where $M$ is the matrix with $a_i$ as its $i$-th column and $D$ is $-d$ times the identity matrix of dimension $p$. Importantly, the set of non-negative integer solutions of any system of homogeneous linear equations with rational coefficients forms an affine semigroup, which can be computed as detailed in \cite{Nsemig}. We denote by $\mathcal{X} \subset \N^{q+p}$ the affine semigroup defined by the set of non-negative integer solutions of the system, with $L = \{l_1, \ldots, l_k\}$ being its minimal generating set. Furthermore, we define $\pi(y)$ as the projection onto the last $p$ coordinates of $y = (\lambda, x) \in \mathcal{X}$.

\begin{lemma}
Let $S$ be an affine semigroup generated by $A=\{a_1,\ldots,a_q\}$, and $d$ a positive integer. The quotient $\frac{S}{d}$ is generated by $\pi (L)=\{ \pi (l_i)\mid i=1,\ldots ,k\}$.
\end{lemma}

\begin{proof}
Get $x\in \frac{S}{d}$, so $dx=\sum_{i=1}^q\lambda_i a_i$ for some $\lambda_1,\ldots,\lambda_q\in \N$. Hence, $(\lambda , x) \in \mathcal{X}$, and there exist $\alpha_1,\ldots ,\alpha _k\in \N$ such that $(\lambda , x)=\sum _{i=1}^k \alpha_i l_i$. Thus, $x=\pi(\sum _{i=1}^k \alpha_i l_i)$. Analogously, if $x$ can be obtained from $\pi (L)$ by a non-negative linear combination of its elements, then $x\in \frac{S}{d}$.
\end{proof}

Algorithm \ref{computeCociente} efficiently computes the minimal generating set of the quotient of any affine semigroup $S$ by a positive integer $d$. We denoted by $I_p$ the identity matrix of dimension $p$.

\begin{algorithm}[H]
\caption{Computing the minimal generating set of $\frac{S}{d}$.}\label{computeCociente}
\KwIn{$A=\{a_1,\ldots ,a_q\}$, the minimal generating set of $S$.}
\KwOut{A minimal generating set of $\frac{S}{d}$.}

$M \leftarrow (a_1\vert \cdots \vert a_q)$\;
$D \leftarrow -dI_p$\;
$L \leftarrow$ the minimal generating set of the $\N$-solutions of $(M\vert D) (\lambda , x)^t = 0$\;
$L \leftarrow \pi(L)$\;
$\mathcal L \leftarrow$ remove non-minimal elements belonging to $L$\;
\Return{$\mathcal L$}
\end{algorithm}

We begin by defining convex body semigroups, a family closed under the quotient by any positive integer. Given a convex body $F$ in $\R^p_{\geq}$ (a compact convex subset with non-empty interior), we define $\CaC(F)$ as the non-negative integer cone spanned by $F$. This cone consists of all elements in $\N^p$ that can be expressed as a non-negative rational linear combination of elements in $F$.
The definition of cone spanned by a set can be extended to every subset of $\R^p_\geq$.

Furthermore, the set $\mathcal{B}(F) = \bigcup_{i=0}^{\infty} iF \cap \N^p$ is an affine semigroup if and only if $F \cap \tau \cap \Q^p \neq \emptyset$ for any extremal ray of the rational cone spanned by $F$ (see \cite{HHT}). In this work, we assume every considered convex body satisfies this property. While not always $\CaC(F)$-semigroups (see \cite{G-S-V-20} and references therein), such semigroups are called {\em convex body (affine) semigroups}.

\begin{lemma}
    Let $F\subset \R^p_\ge$ be a convex body, and $d\in \N\setminus\{0\}$. Then, $\frac{\CaB(F)}{d}$ is a convex body semigroup. Moreover, $\frac{\CaB(F)}{d}=\CaB(\frac{F}{d})$.
\end{lemma}

\begin{proof}
Since $\CaB(F)$ is the convex body semigroup determined by $F$, $x\in \frac{\CaB(F)}{d}$ if and only if $dx\in \CaB(F)$, that is, $x\in \frac{i}{d}F$ for any $i\in \N$. So, $\frac{\CaB(F)}{d}=\bigcup_{i=0}^{\infty} \frac{i}{d}F\cap \N^p=\CaB(\frac{F}{d})$. That proves the lemma.
\end{proof}

The previous result implies that the set of the convex body semigroups is closed under the quotient by any positive integer. In particular, the set of the convex body semigroups generating a fixed integer cone is also closed under the quotient. We summarise these results in two corollaries.

\begin{corollary}
The set of all convex body semigroups is closed under the quotient by any integer.
\end{corollary}

\begin{corollary}
Let $\CaC$ be an integer cone and $\frak{B}$ the of the set of all of convex body semigroups such that $\CaC=\CaC(\CaB)$ for any $\CaB\in \frak{B}$. Hence, $\frak{B}$ is closed under the quotient by any integer.
\end{corollary}

To introduce the Cohen-Macaulay simplicial semigroups, we need to give several definitions. For any finitely generated commutative semigroup $S$, and a field $\k$, the semigroup ring of $S$ over $\k$ is defined. This ring is determined by $\bigoplus_{m \in S} \k \chi^m$ endowed with a multiplication which is $\k$-linear and such that $\chi^m \cdot \chi^n = \chi^{m+n}$, $m$ and $n \in S$ (see \cite{referencia_de_Pilar}). 

Let $R$ be a Noetherian local ring, a finite $R$-module
$M\neq0$  is a Cohen-Macaulay module if ${\rm depth}(M)=\dim(M)$. If $R$ itself is a Cohen-Macaulay module, it is called a Cohen-Macaulay ring (see \cite{libro-C-M}). A semigroup is called Cohen-Macaulay if its associated semigroup ring $\k[S]$ is Cohen-Macaulay.

An affine semigroup $S\subset \N^p$ is simplicial if $\CaC$ has just $p$ extremal rays. We denote by $\{\tau_1,\ldots ,\tau_p\}$ the set of all extremal rays of $S$. The following result is proved in \cite{GSW}, characterising the affine simplicial semigroups.

\begin{theorem}\cite[Theorem 2.2]{GSW}\label{C-M_Rosales}
Let $S\subset \N^p$ be an affine simplicial semigroup. Then, the following conditions are equivalent:
\begin{enumerate}
\item $S$ is Cohen-Macaulay.
\item For any $a,b\in S$ with $a+n_i=b+n_j$ ($1\le i\neq j\le p$), $a-n_j=b-n_i\in S$. Where $n_i$ is an element belonging to $\tau_i\cap S$ for $i=1,\ldots ,p$.
\end{enumerate}
\end{theorem}

Note that Cohen-Macaulay simplicial $\CaC$-semigroup does not exist (see \cite[Corollary 11]{G-M-V-19}).

The above theorem allows us to state a result that can obtain many Cohen-Macaulay semigroups from a given one using the quotient by a positive integer.

\begin{proposition}
Let $S$ be a Cohen-Macaulay affine simplicial semigroup. Then, the affine simplicial semigroup $\frac{S}{d}$ is Cohen-Macaulay for any positive integer $d$.
\end{proposition}

\begin{proof}
Trivially, $S$ simplicial implies $\frac{S}{d}$ is also simplicial. Besides, since the cone associated to $\frac{S}{d}$ and the cone associated to $S$ are the same, $\frac{S}{d}$ is an affine semigroup (see \cite[Corollary 2.10]{Brunz}).

Get $a,b\in \frac{S}{d}$ such that $a+n_i=b+n_j$ for any $1\le i\neq j\le p$, with $n_i\in \tau_i\cap \frac{S}{d}$ and $n_j\in \tau_j\cap \frac{S}{d}$. Hence, $d(a+n_i)=d(b+n_j)$. Since $S$ is Cohen-Macaulay, Theorem \ref{C-M_Rosales} affirms that $d(a-n_j)=d(b-n_i)\in S$, we obtain that $a-n_j=b-n_i\in \frac{S}{d}$. Thus, we conclude that $\frac{S}{d}$ is Cohen-Macaulay.
\end{proof}

\section{Computing $\CaD_d(S)$ and building a tree on an arithmetic variety} 
\label{compD_dTree}
Let $S$ be a $\CaC$-semigroup and $d\in\N\setminus\{0\}$. One of the aims of this section is to describe the set of all the $\CaC$-semigroups $T$ such that $S=\frac{T}{d}$. From now on, this set is denoted by $\CaD_d(S)$, that is, a subset of any arithmetic variety containing $S$. In concordance with the notation given in \cite{RosalesAV}, an {\em arithmetic variety} is a non-empty family of affine semigroups $\CaA$ that fulfils the following conditions:
\begin{enumerate}
    \item If $G, H \in \CaA$, then $G \cap H \in \CaA$.
    \item If $G \in \CaA$ and $d \in \N \setminus \{0\}$, then $\frac{G}{d} \in \CaA$.
\end{enumerate}
Another main objective of this section is to introduce a tree of infinite and finite arithmetic varieties.

We start to provide an explicit description of the $\CaC$-semigroups belonging to $\CaD_d(S)$. Fixed $f\in \CaC$, consider $M_f=\{m\in S\setminus d\N ^p\mid m\prec f \}$, and 
\[\overline{M}_f=\Bigg\{\{\lambda_1,\ldots ,\lambda_q\}\subseteq M_f \mid \sum_{i=1}^q a_i\lambda_i\in dS\sqcup (\N^p\setminus d\N^p), \forall a_1,\ldots ,a_q\in[d-1]
\Bigg\}.\]
Note that $\overline{M}_f$ is finite. By convention, we consider that any linear combination of the elements in the empty set is zero.

\begin{theorem}\label{theorem_D_d}
Let $S$ be a $\CaC$-semigroup and a positive integer $d$. Then, the set $\CaD_d(S)$ is equal to 
\[ \left\{ T(f,\Lambda) \mid f\in \CaC \mbox{ with }f\succeq d\Fb(S),\mbox{ and } \Lambda\in \overline{M}_f\cup \{\{\emptyset\}\} \right \},\]
with 
\begin{multline*}
    T(f,\Lambda)=\{x\in \CaC\mid x\succ f\} \cup\\ \left(dS +\left\{\sum_{i=1}^{\sharp \Lambda} a_i\lambda_i\mid a_i \in [d-1]\mbox{ and } \lambda_i\in\Lambda,\, \forall i\in \{1,\ldots, \sharp \Lambda\}\right\}  \right).
\end{multline*}

\end{theorem}

\begin{proof}
Fix $f\in \CaC$ with $f\succ d\Fb(S)$,  and $\Lambda=\{\lambda_1,\ldots ,\lambda_q\}\in \overline{M}_f$. Let us prove that $T(f,\Lambda)$ is a $\CaC$-semigroup. Since $0=d0+\sum_{i=1}^q 0\lambda_i$, we have $0\in T(f,\Lambda)$. Let $x,y\in \{x\in \CaC\mid x\succ f\}$, and $z=2s+\sum_{i=1}^{q} a_i\lambda_i$ for some $s\in S$ and $a_1,\ldots, a_q\in[d-1]$, note that, $x+y,x+z,y+z\in \{x\in \CaC\mid x\succ f\}$. Consider $x,y\in T(f,\Lambda)$ such that $x,y\prec f$, therefore, $x=ds+\sum_{i=1}^{q} a_i\lambda_i$ and $y=ds'+\sum_{i=1}^{q} a'_i\lambda_i$ with $s,s'\in S$ and $a_1,a'_1\ldots, a_q,a_q'\in[d-1]$. Since $\Lambda\in \overline{M}_f$, we can express $x+y=d(s+s')+ds''+\sum_{i=1}^{q} a''_i\lambda_i$ for some $a''_1,\ldots,a''_q\in[d-1]$ and $s''\in S$. So, we deduce that  $T(f,\Lambda)$ is closed under addition. Observe that, $\{x\in \CaC\mid x\succ f\}\subset  T(f,\Lambda)$, thus $\CaC\setminus T(f,\Lambda)$ is finite. Hence, $T(f,\Lambda)$ is a $\CaC$-semigroup.

Since $dS \subset  T(f,\Lambda)$, we obtain $S\subset \frac{T(f,\Lambda)}{d}$. Let $x\in \frac{T(f,\Lambda)}{d}$, by definition $dx\in T(f,\Lambda)$. If $dx\succ f\succ d\Fb(S)$, then $x\succ \Fb(S)$, and $x\in S$. In case that $dx=ds+\sum_{i=1}^{q} a_i\lambda_i$ for any $s\in S$ and $a_1,\ldots, a_q\in[d-1]$, and taking into account that $\Lambda\in \overline{M}_f$, we deduce $\sum_{i=1}^{q} a_i\lambda_i\in dS$. Therefore, $S=\frac{T(f,\Lambda)}{d}$. 

Let $T$ be a $\CaC$-semigroup such that $S=\frac{T}{d}$, and consider  $\Lambda=M_f\cap T=\{\lambda_1,\ldots ,\lambda_q\}$ with $f=\Fb(T)$. We now prove that $T=T(f,\Lambda)$. Note that it is equivalent to show that $\Lambda \in \overline{M}_f$. Take some $a_1,\ldots, a_q\in[d-1]$. In case that $\sum_{i=1}^{q} a_i\lambda_i\notin d\N^p$, we hold the result. Otherwise, assume that $\sum_{i=1}^{q} a_i\lambda_i\in d\N^p\setminus dS$. Thus, $\frac{\sum_{i=1}^{q} a_i\lambda_i}{d}\notin S$. Since $S=\frac{T}{d}$, it follows that $d\frac{\sum_{i=1}^{q} a_i\lambda_i}{d}\notin T$, contradicting that $\Lambda \subset T$.
\end{proof}

Observe that $T(f,\emptyset)= \{x\in \CaC\mid x\succ f\} \cup dS$. Moreover, if $f=d\Fb(S)$, then $T(f,\emptyset)$ is the one with the largest genus of the $\CaC$-semigroups in $\CaD_d(S)$ with Frobenius element equal to $d\Fb(S)$.

From Theorem \ref{theorem_D_d}, the set $\CaD_d(S)$ is not finite  since $f$ can be any element belonging to $\CaC$ greater than $d\Fb(S)$. But, fixed an element $f\succ d\Fb(S)$, the set $\CaD_d(S,f)=\{T\in \CaD_d(S)\mid \Fb(T)\preceq f\}$ is finite. In particular, $\CaD_d(S,f)$ is equal to $\{T(f,\Lambda)\mid \Lambda \in \overline{M}_f \cup  \{\{\emptyset \}\} \}$.

\begin{example}\label{ex_cociente_d}
Consider the $\CaC$-semigroup minimally generated by
\begin{equation}\label{semigrupo_ejemplo1}
\{(4, 1), (5, 2), (7, 2), (9, 5), (4, 2), (6, 2), (6, 3), (7, 3), (11, 6)\}.\end{equation}
This semigroup is the non-negative integer cone $\CaC$ determined by the extremal rays $\{(4,1),(9,5)\}$ except the points $\{(2,1),(3,1)\}$ (see Figure \ref{fig1_ex_cociente_d}).
\begin{figure}[h]
    \centering
   \includegraphics[scale=.3]{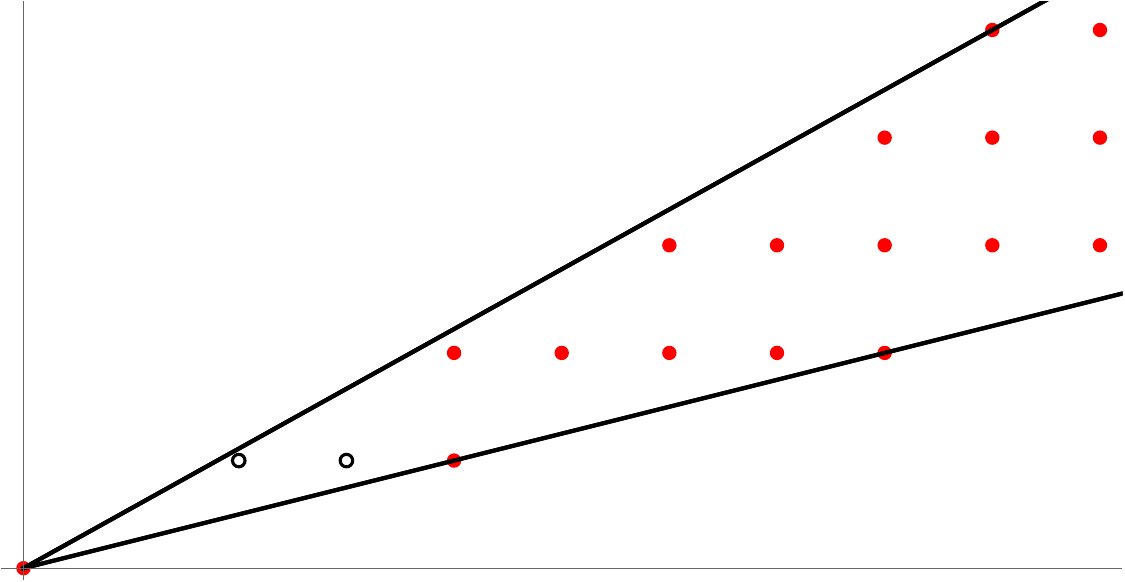}
     \caption{ $\CaC$-semigroup minimally generated by \eqref{semigrupo_ejemplo1}.}
    \label{fig1_ex_cociente_d}
\end{figure}

Fixed $d=3$, $f=3\Fb(S)=(9,3)$, and the reverse graded lexicographical order (see \cite{Cox}), we obtain that the set $M_f$ is equal to
$$\{(4,1),(4,2), (5, 2), (6,2), (7, 2), (7,3), (8, 2), (8,3)\},$$
and $T(f,\emptyset)$ is the $\CaC$-semigroup showed in Figure \ref{fig2_ex_cociente_d}.
\begin{figure}[h]
    \centering
   \includegraphics[scale=.3]{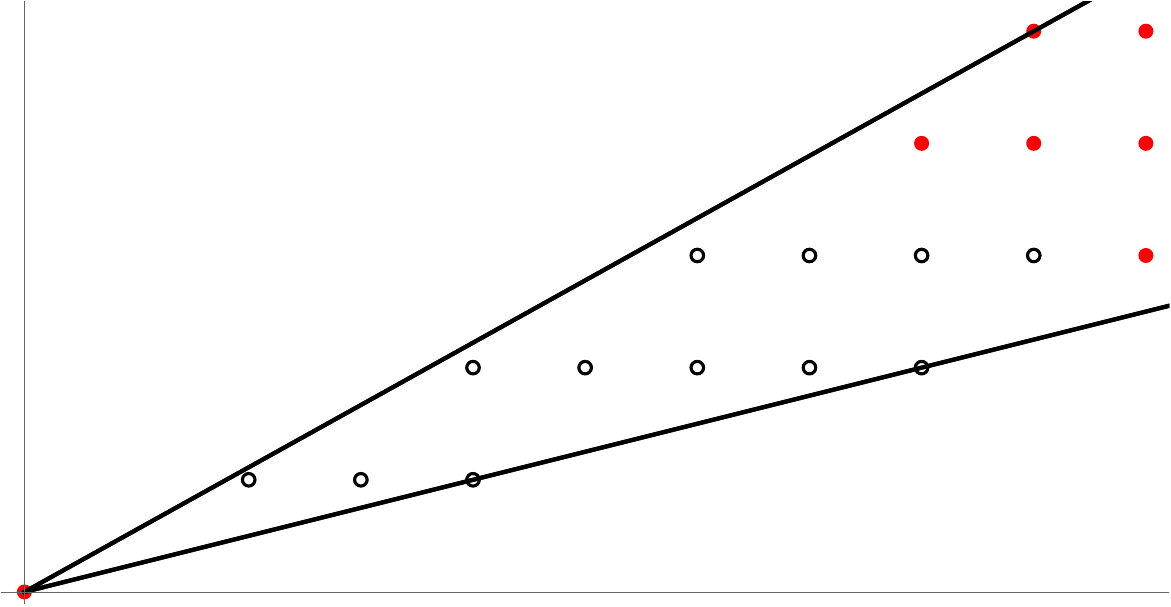}
     \caption{ $T(f,\emptyset)$.}
    \label{fig2_ex_cociente_d}
\end{figure}
Moreover, the cardinality of $\CaD_3(S,f)$ is bounded by $\sum_{i=0}^8 \binom{8}{i}=256$, which is the amount of all the possible elements in $\overline{M}_f$ plus one.

Since, $(4,1)+(5,2)\in 3\N^2$, no element in $\overline{M}_f$ can contain $\{(4,1),(5,2)\}$, and, taking into account $(4,1)+(4,2), 2(4,1)\in M_f$, we have to remove those redundant $\CaC$-semigroups. For example, $T(f,\{(4,1)\})=T(f,\{(4,1),(8,2)\})$. So, we obtain the cardinality of $\CaD_3(S,f)$ is just equal to $151$. Figure \ref{table_ex_cociente_d} shows some graphical examples of elements belonging to $\CaD_3(S,f)$.
\begin{figure}[h]
    \centering
$\begin{array}{cc}
       \includegraphics[scale=.25]{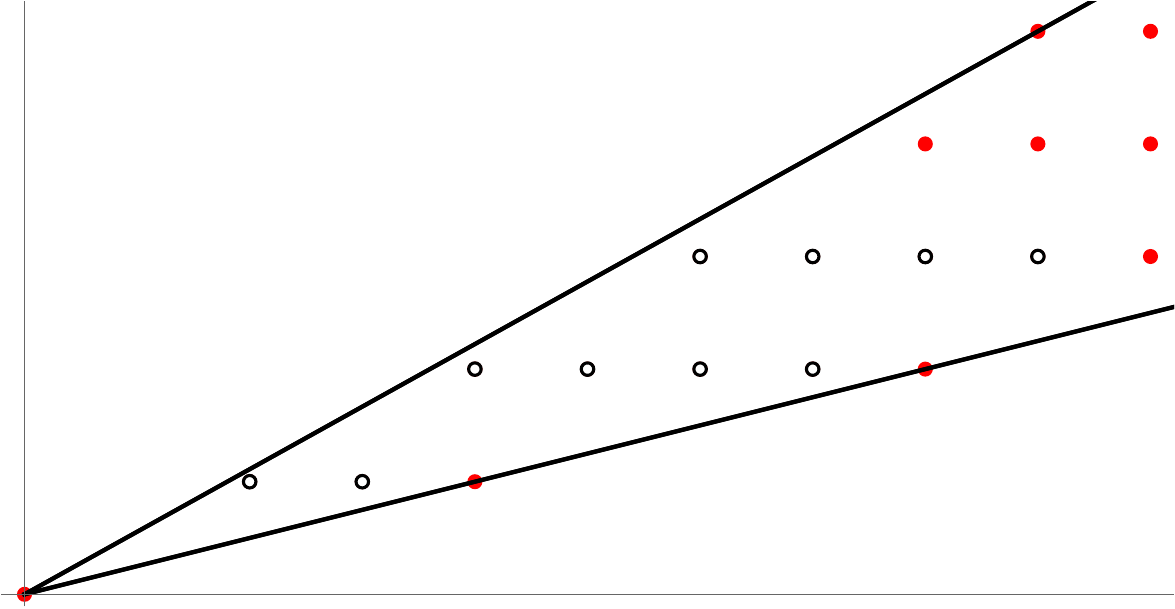}  & \includegraphics[scale=.25]{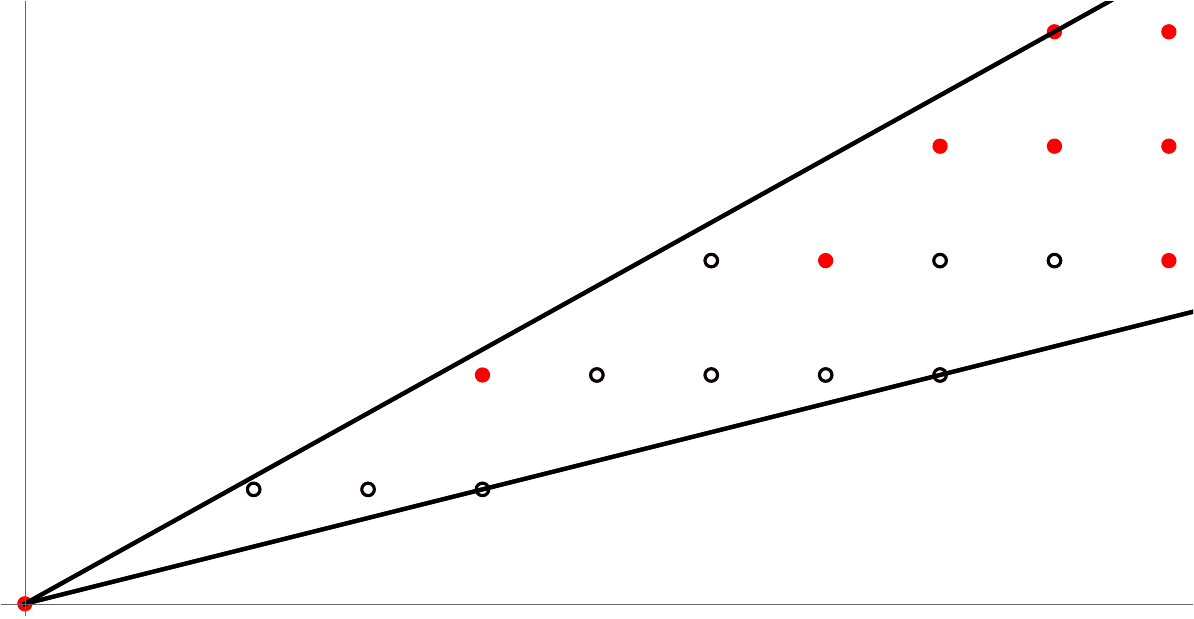} \\
       \Lambda=\{(4,1)\} & \Lambda=\{(4,2),(7,3)\} \\ 
       \\
        \includegraphics[scale=.25]{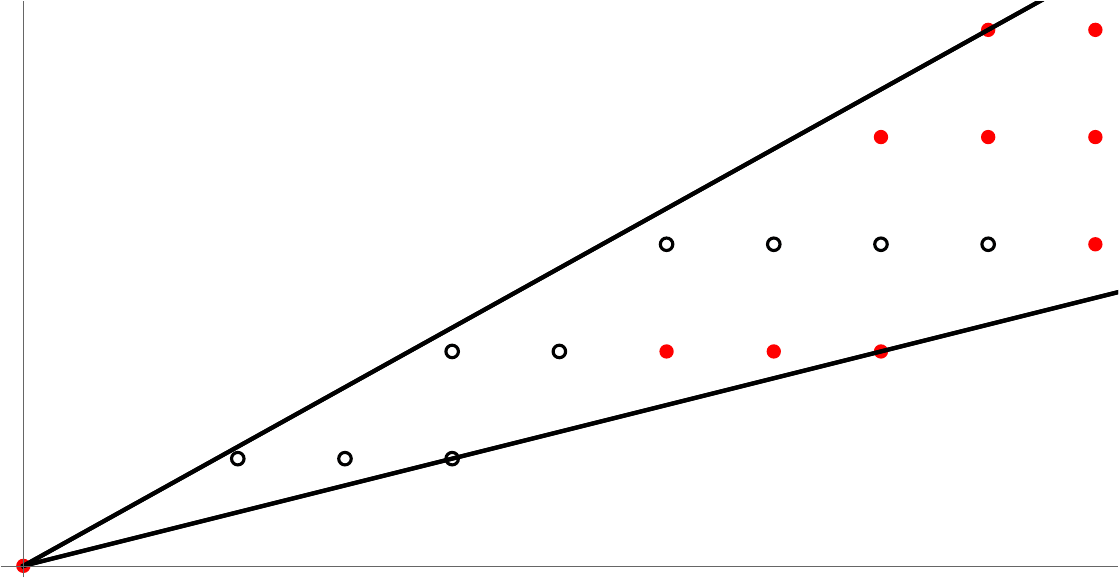} & \includegraphics[scale=.25]{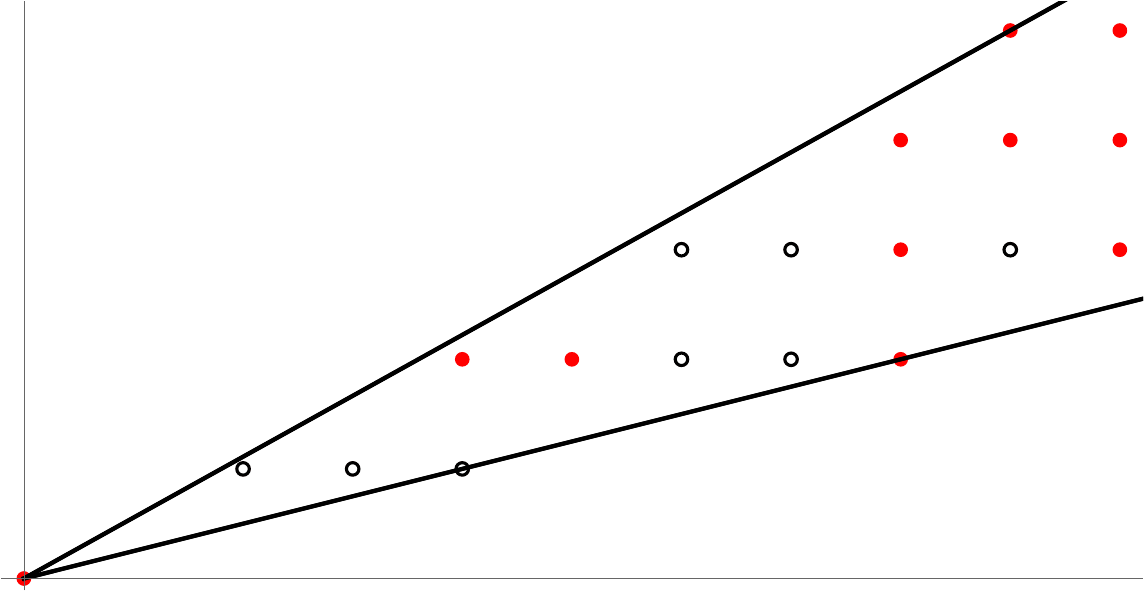} \\
        \Lambda=\{(6,2),(7,2),(8,2)\} & \Lambda=\{(4,2),(5,2),(8,2),(8,3)\} \\ 
\end{array}$
    \caption{Some elements in $\CaD_3(S)$.}
    \label{table_ex_cociente_d}
\end{figure}

\end{example}

This section ends by introducing a tree on an arithmetic variety of $\CaC$-semigroups. Consider $\CaA$ an arithmetic variety, we can build the tree $G_{\CaA,d}$ with root $\CaC$, whose vertex set is the arithmetic variety, and $(T, S)\in \CaA^2$ is a directed edge if and only if $S=\frac{T}{d}$. It is more when $(T, S)$ is an edge, we say that $T$ is a child of $S$. 

\begin{lemma}
\label{tree}
For any $\CaA$ arithmetic variety of $\CaC$-semigroups, $G_{\CaA,d}$ is a tree with root $\CaC$. Furthermore, the set of children of any $S\in \CaA$ is the set $\CaD_d(S)\cap \CaA$.
\end{lemma}

\begin{proof}
 Let $S\in \CaA$ such that $S\ne \CaC$. Consider the sequence $\{S_i\}_{i\in \N}$ defined by $S_i=\frac{S}{d^i}$. Since the complement of $S$ in $\CaC$ is finite, there exists a positive integer $n$ such that $S=S_0\subset S_1\subset \cdots \subset S_n=\CaC$. Therefore, for every vertex $S\in\CaA$, a path exists connecting $S$ and $\CaC$. The uniqueness of the path is given by its definition.
\end{proof}

Note that the graph $G_{\CaA,d}$ could be infinite. If $d=2$, and we consider the set of all numerical semigroups as the arithmetic variety, this result is known (see \cite{OR2020}). 

Given $f\in \CaC$,  take the set $\CaA_f=\{S\in \CaA\mid \Fb(S)\preceq f\}$. Let $S$ and $T$ be two $\CaC$-semigroups and $d$ be a positive integer. Observe that $\Fb(S\cap T)=\max_\preceq\{\Fb(S),\Fb(T)\}$, and, taking into account that $S\subseteq \frac{S}{d}$, it follows $\Fb(\frac{S}{d})\preceq\Fb(S)$. From the previous considerations, we deduce that $\CaA_f$ is an arithmetic variety. Since a finite amount of $\CaC$-semigroups with Frobenius element less than or equal to $f$ exists, $\CaA_f$ is also finite. Therefore, we can fully construct a graph $G_{\CaA_f,d}$. In view of Lemma \ref{tree}, to build the graph $G_{\CaA_f,d}$,
it is enough to compute $\CaD_d(S)\cap \CaA_f=\CaD_d(S,f)\cap \CaA=\{T(f,\Lambda)\in \CaA \mid \Lambda\in \overline{M}_f\cup\{ \{ \emptyset\}\}\}$, as the following example illustrates.

\begin{example}
As Example \ref{ex_cociente_d}, let $\CaC$ be again the integer cone determined by the extremal rays $\{(4,1),(9,5)\}$, $\preceq$ the reverse lexicographical order, and $f=(4,2)$. Consider $\CaA$ the arithmetic variety consisting of all $\CaC$-semigroups, hence, $\CaA_f$ contains twelve elements: $S_0=\CaC\setminus\{\emptyset\}$, $S_1=\CaC\setminus\{(2,1)\}$, $S_2=\CaC\setminus\{(3,1)\}$, $S_3=\CaC\setminus\{(4,1)\}$, $S_4=\CaC\setminus\{(2,1),(3,1)\}$, $S_5=\CaC\setminus\{(2,1),(4,1)\}$, $S_6=\CaC\setminus\{(3,1),(4,1)\}$, $S_7=\CaC\setminus\{(2,1),(4,2)\}$, $S_8=\CaC\setminus\{(2,1),(3,1),(4,1)\}$, $S_9=\CaC\setminus\{(2,1),(3,1),(4,2)\}$, $S_{10}=\CaC\setminus\{(2,1),(4,1),(4,2)\}$, and $S_{11}=\CaC\setminus\{(2,1),(3,1),(4,1),(4,2)\}$. Figure \ref{grafo_finito} illustrates the finite graph $G_{\CaA_{(4,2)},2}$.
\begin{figure}[h]
\centering
\tikz [ my tree ]
  \node {$S_0=\CaC$}
    child {node {$S_1$}
        child {node {$S_{7}$}}
        child {node {$S_{9}$}}
        child {node {$S_{10}$}}
        child {node {$S_{11}$}}
    }
    child {node {$S_2$}}
    child {node {$S_3$}}
    child {node {$S_4$}}    
    child {node {$S_5$}}
    child {node {$S_6$}}
    child {node {$S_8$}}
;
     \caption{ $G_{\CaA_{(4,2)},2}$.
             }
    \label{grafo_finito}
\end{figure}

\end{example}

\section{Some results about irreducibility over the quotients of $\CaC$-semigroups}
\label{irre}

This section establishes two key results regarding $\CaC$-semigroups. Firstly, we demonstrate that each $\CaC$-semigroup constitutes precisely one-half of infinitely many distinct symmetric $\CaC$-semigroups. Secondly, we prove that every $\CaC$-semigroup forms one-fourth of a unique pseudo-symmetric $\CaC$-semigroup. 

Denote by $\mathcal{O}$ the set $\N\setminus2\N$ of odd non-negative integers.

Inspired by \cite[Theorem 6.7]{libroRosales}, and by the study of the irreducibility of $\CaC$-semigroups given in \cite{SomepropCsemgp}, we have the following result.

\begin{theorem}\label{thrmQuoSim}
    Let $S$ be a $\CaC$-semigroup and $f\in\CaC\cap\mathcal{O}^p$ such that $f-f_i-f_j\in S$ for all $f_i,f_j\in PF(S)=\{f_1, \ldots,f_t\}.$
    Then,
    \begin{eqnarray*}
      T&=&2S\cup\{x\in \CaC\mid f-x\notin\CaC\}\\
      &\cup & \bigcup_{i=1}^t \Big((f-2f_i)+2S\Big)\\
      &\cup&\left\{x\in\CaC\setminus(2\N^p\cup\mathcal{O}^p)\mid x \succ\tfrac{f}{2},\,f-x\in\CaC\right\}  
    \end{eqnarray*}
    is a symmetric $\CaC$-semigroup, with Frobenius element $f$, and such that $S=\frac{T}{2}.$
\end{theorem}
\begin{proof}
We start by proving that $T$ is a $\CaC$-semigroup. To show that $T$ is closed under addition, we distinguish the following cases: 
\begin{itemize}
    \item The sum of two elements of $2S$ belongs to $2S$.
    
    \item If $x,y\in\CaC$ such that $f-x,f-y\notin\CaC$, then $f-(x+y)\notin\CaC$, otherwise $f-x=y+c\in \CaC$ for some $c\in \CaC$, which is impossible.
    
    \item For any $f_i,f_j\in PF(S)$, we have that  $(f-2f_i+2S)+(f-2f_j+2S)=2(f-f_i-f_j)+2S\subset 2S$.
    \item If $x,y\in\left\{x\in\CaC\setminus(2\N^p\cup\mathcal{O}^p)\mid x\succ\tfrac{f}{2},\,f-x\in\CaC\right\}$, then $f-(x+y)\preceq 0$, and this implies that $f-(x+y)\notin \CaC$.
    
    \item Let $x\in 2S$ and $y\in \{x\in \CaC\mid f-x\notin\CaC\}$. Hence, $x+y\in \{x\in \CaC\mid f-x\notin\CaC\}$, otherwise, there exists $c\in \CaC$ such that $f-y=x+c\in \CaC$, which it is not possible.
    
    \item For any $s\in 2S$ then $2s+((f-2f_i)+2S)\subset (f-2f_i)+2S$ for every $i\in[t]\setminus\{0\}$.
    
    \item Let $x\in 2S$ and $y\in \left\{x\in\CaC\setminus(2\N^p\cup\mathcal{O}^p)\mid x \succ\tfrac{f}{2},\,f-x\in\CaC\right\}$. So, if $f-(x+y)\notin \CaC$, then $x+y\in \{x\in \CaC\mid f-x\notin\CaC\}$. In the other case, if $f-(x+y)\in \CaC$, since $x+y\in \CaC\setminus(2\N^p\cup\mathcal{O}^p)$, and $x+y\succ y \succ f/2$, we have  $x+y \in \left\{x\in\CaC\setminus(2\N^p\cup\mathcal{O}^p)\mid x\succ\tfrac{f}{2},\,f-x\in\CaC\right\}$.
    
    \item Consider $x\in \{x\in \CaC\mid f-x\notin\CaC\}$ and $y\in \bigcup_{i=1}^t \left((f-2f_i)+2S\right) \cup \left\{x\in\CaC\setminus(2\N^p\cup\mathcal{O}^p)\mid x \succ\tfrac{f}{2},\,f-x\in\CaC\right\}$. Observe that $x+y\in \{x\in \CaC\mid f-x\notin\CaC\}$. Otherwise, $f-x=y+t\in \CaC$ for some $t\in \CaC$, which is not possible.
    
    \item If $y\in \bigcup_{i=1}^t \left((f-2f_i)+2S\right)$ and $x\in\left\{x\in\CaC\setminus(2\N^p\cup\mathcal{O}^p)\mid x \succ\tfrac{f}{2},\,f-x\in\CaC\right\}$  then, $x+y\in \{x\in \CaC\mid f-x\notin\CaC\}\cup\left\{x\in\CaC\setminus(2\N^p\cup\mathcal{O}^p)\mid x \succ\tfrac{f}{2},\,f-x\in\CaC\right\}$.
\end{itemize} 
Thus, we conclude that $T$ is a semigroup of $\N^p$. Besides, since $\CaH(T)\subset \{x\in\CaC\mid f-x\in \CaC\}$, this set is finite, and $Fb(T)=f$.

Now, let us prove that $S=\frac{T}{2}$. Trivially, $S\subset \frac{T}{2}$. For the other inclusion, we first show that if $x\in\CaH(S)\cup2\CaH(S)$, then $f-x\in \CaC$. Let $x\in \CaH(S)$. From Proposition \ref{carcH(S)} there exists $f_i\in PF(S)$ such that $f_i-x\in S$, so $f-x=f-f_i+f_i-x$. If $f-f_i\notin S$, and by hypothesis $f-f_i-f_j\in S$ with $f_j\in PF(S)$, it would lead $f-f_i=f_j+s\notin S$, for some $s\in S$, contradicting the definition of $f_j$. Hence, $f-x=f-f_i+f_i-x\in S\subseteq \CaC$. 
Analogously, if $x\in 2\CaH(S)$, then $x=2y$ with $y\in\CaH(S),$ again by Proposition \ref{carcH(S)} we obtain that $f-x=f-2f_i+2(f_i-y)\in S$. Whence, $f-x\in \CaC$. Let $x\in \frac{T}{2}$, so $2x\in T$. Since $2x\in 2\N^p$, $2x$ have to belong to $2S\cup \{x\in \CaC\mid f-x\notin\CaC\}$. If $2x\in2S$, we have done. Assume that $2x\in\CaC\setminus 2S$ and $f-2x\notin\CaC$, hence $2x\notin \CaH(S)\cup 2\CaH(S)$. Thus, $x\in S$. Therefore, $S=\frac{T}{2}.$

To prove that $T$ is symmetric, take $x\in \CaC$. By applying Proposition \ref{caracSIM}, we must show that $x\in \CaH(T)$ if and only if $f-x\in T.$ We distinguish three cases depending on the parity of $x$:
    \begin{itemize}
        \item If $x\in 2\N^p$ then as $x\notin T,$ we have that $\frac{x}{2}\notin S.$ In view of Proposition \ref{carcH(S)} there exists $f_i\in PF(S)$ such that $f_i- \frac{x}{2}\in S$. Thus,  $2f_i-x\in 2S$, hence $f-x=f-2f_i+2f_i-x\in T.$
        \item If $x\in\mathcal{O}^p$, then $f-x\in 2\N^p.$ Thus, if $f-x\notin T$, by using the preceding case, we obtain that $f-(f-x)=x\in T$. A contradiction.
        \item If $x\in\CaC\setminus\{2\N^p\cup\mathcal{O}^p\}$, then  $x\prec \frac{f}{2}$, otherwise, since $x\notin T$, we get that $f-x\in \CaC$ and $f-x\notin \CaC,$ which is impossible. Observe that  $f-x\notin 2\N^p\cup\mathcal{O}^p$ and $f-x\succ\frac{f}{2}$. Hence, as $x\notin T$, we know that $f-x\in \CaC$. Therefore, $f-x\in T$.
    \end{itemize}
Conversely, let $x\in T$, we prove that $f-x\notin T$. We distinguish four cases depending on $x$:
    \begin{itemize}
        \item Let $x=2s$ for some $s\in S$. Since $f-2s\in \mathcal{O}^p$ it follows that $f-x\notin 2S$. Besides, $f-(f-x)=x\in \CaC$. If $f-2s=f-2f_i+2s'$ for some $f_i\in PF(S)$ and $s'\in S$, then $2f_i=2s''$, with $s''=s+s'\in S$ and thus, $f_i\in S$ which is not possible. Therefore, $f-x\notin T$.
        
        \item If $x\in \{x\in \CaC\mid f-x\notin \CaC\}$, then $f-x\notin T$.
        
        \item Let $x=f-2f_i+2s$ for some $f_i\in PF(S)$ and $s\in S$, and assume that $f-x\in \CaC$. Notice that, $f-x=2f_i-2s=2(f_i-s)\notin 2S$. If $f-x=f-2f_j+2s'$ for some $s'\in S$ and $f_j\in PF(S)$, then $x\in 2\N^p$, but $x=f-2f_i+2s\notin 2\N^p$. Considering that $f-x\in 2\N^p$, we conclude that $f-x\notin T$.
        
        \item If $x\in\left\{x\in\CaC\setminus(2\N^p\cup\mathcal{O}^p)\mid x \succ\tfrac{f}{2},\,f-x\in\CaC\right\}$, then $f-x\in \CaC\setminus (2\N^p\cup\mathcal{O}^p)$ and $f-x\prec \frac{f}{2}$. Notice that if $f-x=f-2f_i+2s$ for some $f_i\in PF(S)$ and $s\in S$, we obtain that $x=2f_i-2s$, which contradicts the fact that $x\notin 2\N^p$.
    \end{itemize}
\end{proof}

\begin{remark}\label{remarkInfSim}
Due to Theorem \ref{thrmQuoSim}, we can choose infinitely many $T$ for every semigroup $S$. This fact implies that there exist infinitely many symmetric $\CaC$-semigroups $T$ such that $S =\frac{T}{2}$, which is also true to numerical semigroups (\cite[Corollary 6.8]{libroRosales}).
\end{remark}

The following example provides a symmetric $\CaC$-semigroup obtained from Theorem \ref{thrmQuoSim} for the $\CaC$-semigroup minimally generated by $\eqref{semigrupo_ejemplo1}$.

\begin{example}\label{ex_cociente_2_simetrico}
Let $S$ be the $\CaC$-semigroup given in Example \ref{ex_cociente_d}, and consider $f=(13,5)\in \CaC\cap\mathcal{O}^2$. Trivially, the set $PF(S)=\{f_1=(2,1),f_2=(3,1)\}$, and $f-2f_1,f-2f_2,f-f_1-f_2\in S$. Figure \ref{figure_ex_cociente_2_simetrico} shows the symmetric $\CaC$-semigroup $T$ obtained from Theorem \ref{thrmQuoSim} for $S$ and $f$. This $\CaC$-semigroup is minimally generated by
{\small
\begin{multline*}
\{(6, 3), (7, 3), (8, 2), (8, 3), (8, 4), (9, 3), (9, 4), (9, 5), (10, 3), (10, 4),\\ (10, 5), (11, 3), (11, 4), (11, 5), (11, 6), (12, 3), (12, 4), (12, 5), (13, 4),\\ (13, 7), (14, 4), (15, 4)\}.
\end{multline*}
}
\begin{figure}[h]
    \centering
   \includegraphics[scale=.35]{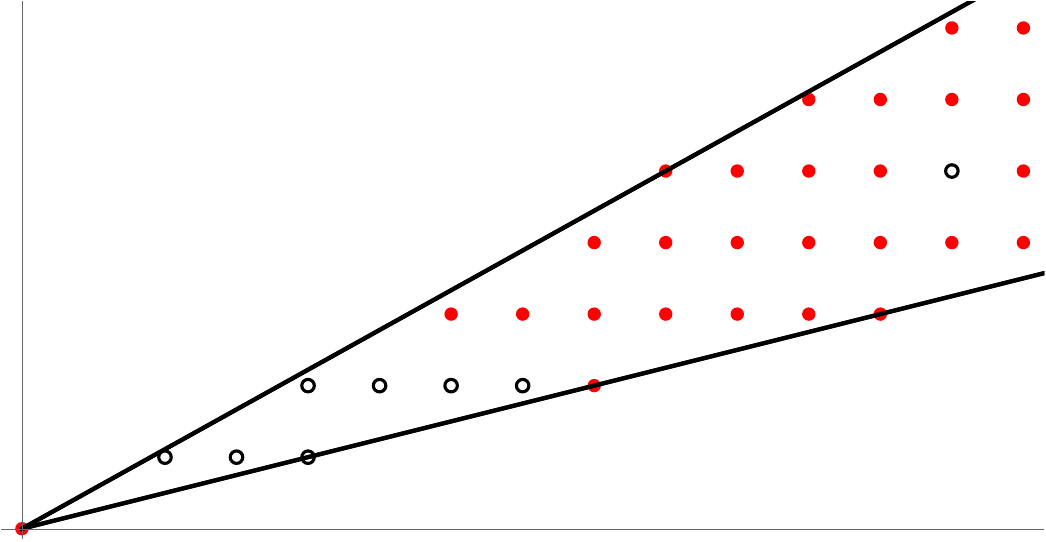}
     \caption{A symmetric semigroup $T$ with $S=\frac{T}{2}$ obtained from Theorem \ref{thrmQuoSim}.}
    \label{figure_ex_cociente_2_simetrico}
\end{figure}

\end{example}

Let us see there is not a parallelism between the symmetric and pseudo-symmetric cases. Previously, we can extend to $\CaC$-semigroups Lemma 6.9 appearing in \cite{libroRosales}.

\begin{lemma}\label{fbQuo}
    Let $S$ be $\CaC$-semigroup which Frobenius element belongs to $2\N^p$. Then
    \[
    \Fb\left(\frac{S}{2}\right) = \frac{\Fb(S)}{2}.
    \]
\end{lemma}
\begin{proof}
It follows by Proposition \ref{propFG}, and by considering that the Frobenius element of a $\CaC$-semigroup is the maximum of the fundamental gaps.
\end{proof}

Note that if $S$ is a  $\CaC$-semigroup and T is a pseudo-symmetric $\CaC$-semigroup such that $S=\frac{T}{2}$, then from the Lemma \ref{fbQuo} we deduce that $\Fb(T) = 2\Fb(S)$ and there exist finitely many $\CaC$-semigroups with Frobenius element $2\Fb(S)$. Hence, we cannot obtain a result similar to Remark \ref{remarkInfSim} for the pseudo-symmetric case.

\begin{proposition}\label{propTpseusT/2irr}
    Let $T$ be a pseudo-symmetric $\CaC$-semigroup. Then, $\frac{T}{2}$ is an irreducible $\CaC$-semigroup.
\end{proposition}
\begin{proof}
    Let $S = \frac{T}{2}$ and suppose that $PF(S) = \{f_1,  \ldots,  f_t\}$. Assume that $\Fb(S)=f_t.$  Thus, by applying Lemma \ref{fbQuo}, we get $\Fb(T) = 2f_t$, and $PS(T)=\{f_t,2f_t\}$. Given any $i\in \{1,\ldots ,t\}$, since $f_i \in \CaH (S)$, then $2f_i \in \CaH (T)$. From Proposition \ref{carcH(S)}, and considering that $T$ is a pseudo-symmetric $\CaC$-semigroup, we have that either $2f_t - 2f_i \in T$ or $f_t-2f_i\in T$. If $2f_t - 2f_i \in T$, then $f_t - f_i \in S$, and this is possible if and only if $f_i = f_t$. If $f_t-2f_i\in T$, then $f_t=2f_i+x$ with $x\in T\subset S$. When $x \neq 0$,  we obtain that $f_t\in S$, which is impossible. Hence, $x=0$ and $f_i=f_t/2$. This proves that $PF(S) \subseteq \{\Fb(S), \frac{\Fb(S)}{2} \}$.
    Therefore, $S$ is either a symmetric or a pseudo-symmetric $\CaC$-semigroup. In both cases, $S$ is an irreducible $\CaC$-semigroup.
\end{proof}

In particular, it is not true that any $\CaC$-semigroup is one-half of a pseudo-symmetric $\CaC$-semigroup.

We can sharpen the preceding result a bit further to distinguish which cases one-half of a pseudo-symmetric $\CaC$-semigroup is symmetric or pseudo-symmetric.

\begin{corollary}
    \label{coroTpsdo-Sirre}
    Let $T$ be a pseudo-symmetric $\CaC$-semigroup and $S=\frac{T}{2}$. Then:
    \begin{enumerate}
        \item $S$ is symmetric if and only if $\Fb(S) \not\equiv 0 \mod 4$.
        \item $S$ is pseudo-symmetric if and only if $\Fb(S) \equiv 0 \mod 4$.
    \end{enumerate}
\end{corollary}

\begin{proof}
We prove the first item. Suppose that $S=\frac{T}{2}$ a pseudo-symmetric $\CaC$-semigroup, then $PF(S)=\{\Fb(S),\frac{\Fb(S)}{2}\}$. By Lemma \ref{fbQuo} we know that $\Fb(S)=\frac{\Fb(T)}{2}$, hence, $\Fb(S) \equiv 0 \mod 4$. Conversely, from Proposition \ref{propTpseusT/2irr} and Lemma \ref{fbQuo}, we deduce that $\{\frac{\Fb(T)}{2},\frac{\Fb(T)}{4}\}\subseteq PF(S).$ In order to prove the other inclusion, take $f\in PF(S)$, then $f\notin S$ and thus $2f\notin T.$ From Proposition \ref{carcH(S)}, we get that $\Fb(T)-2f\in T$ or $\frac{\Fb(T)}{2}-2f\in T$. If $\Fb(T)-2f\in T$, then $2(\Fb(S)-f)\in T$, hence  $\Fb(S)-f\in S$, and this is possible if and only if $f=\frac{\Fb(S)}{2}=\frac{\Fb(T)}{4}$. Analogously to the proof of Proposition \ref{propTpseusT/2irr}, if $\frac{\Fb(T)}{2}-2f\in T$, then $f=\frac{\Fb(S)}{2}=\Fb(T)$. Therefore $PF(S)=\{\Fb(S),\frac{\Fb(S)}{2}\}$.

For the second item, we know that $S$ is irreducible (Proposition \ref{propTpseusT/2irr}), and $S$ is symmetric if and only if $S$ is not pseudo-symmetric. By the previous item, this is possible if and only if $\Fb(S) \equiv 0 \mod 4$.
\end{proof}

To finish this section, we characterise irreducible $\CaC$-semigroups relative to one-half of $\CaC$-semigroups.
Previously, we proved that any $\CaC$-semigroup can be expressed as infinitely many one-fourth of pseudo-symmetric $\CaC$-semigroups.

\begin{lemma}\label{lemaSsymTpsdo}
    Let $S$ be an irreducible $\CaC$-semigroup, $\preceq$ a monomial order and the set
    \[
    A=\{ x\in (\N^p\setminus 2\N^p)\cap \CaC \mid \Fb(S)\prec x\prec 2\Fb(S)\}.
    \]
    Then, 
    \[T=2S \cup A \cup\{x\in \CaC \mid x\succ 2\Fb(S)\}\]
    is a pseudo symmetric $\CaC$-semigroup, with Frobenius element $2\Fb(S)$, and such that $S=\frac{T}{2}$.
\end{lemma}

\begin{proof}
    Observe that 2S is closed under addition, and for any $a,a'\in A$, we have that $a+a'\in \{x\in \CaC \mid x\succ 2\Fb(S)\}$. If $x\in 2S$ and $a\in A$, then $x+a\in A \cup \{x\in \CaC \mid x\succ 2\Fb(S)\}$. Given $x\in \{x\in \CaC \mid x\succ 2\Fb(S)\}$ and $t\in T$ then, $x+t\in \{x\in \CaC \mid x\succ 2\Fb(S)\}$. Therefore, $T$ is a $\CaC$-semigroup. Besides, by construction of $T$, $\Fb(T)=2\Fb(S)$.
    
    Next we show that $S=\frac{T}{2}$. Let $x\in \frac{T}{2}$, then $2x\in T$. Since the elements of $A$ are not in $2\N^p$, $2x\in 2S \cup\{x\in \CaC \mid x\succ 2\Fb(S)\}$. If $2x\in 2S$ then $x\in S$. Otherwise, $2x\succ 2\Fb(S)$ and thus $x\in S$.
    Conversely, if $x\in S$, then $2x\in 2S\subset T$ so, $x\in \frac{T}{2}$.

    Now, we prove that $T$ is pseudo-symmetric. Let $x\in\CaC$, applying Proposition \ref{caracPSEUDOSIM}, it is enough to prove that $x\in T$ if and only if $\Fb(T)-x\notin T$ and $x\neq \frac{\Fb(T)}{2}$. Trivially, $x\in T$ implies $x\neq \frac{\Fb(T)}{2}$. Suppose that $x\in T$, we distinguish cases depending on $x$: 
    \begin{itemize} 
        \item If $x\in 2S$, then $x=2s$ for some $s\in S$. In the case that $\Fb(T)-x=2(\Fb(S)-s)$ belongs to $T$, we have $\Fb(S)-s\in S$, which is not possible since $S$ is irreducible.
        
        \item Let $x\in A$ and assume that $\Fb(T)-x\in T$. Since  $\Fb(T)-x\notin 2\N^p$, it follows that $\Fb(T)-x\succ \Fb(T)$ or $\Fb(T)-x\succ \Fb(S)$. If $\Fb(T)-x\succ \Fb(T)$ we deduce that $-x\succ 0$, which is not possible. If $\Fb(T)-x\succ \Fb(S)$, then $\Fb(S)-x\succ 0$, this implies that $\Fb(S)\succ x$, a contradiction, by hypothesis $x\in A$. We conclude that $\Fb(T)-x\notin T$.
        \item If $x\in \{x\in \CaC \mid x\succ 2\Fb(S)\}$, then $\Fb(T)-x\prec 0$ and thus $\Fb(T)-x\notin T$.
    \end{itemize}
    
    Conversely, let $x\in \CaC\setminus\{\Fb(S)\}$ and $\Fb(T)-x\notin T$. Let us show that $x\in T$. If $x\in 2\N^p$, then $x=2z$, for some $z\in \CaC$. Therefore, $\Fb(T)-x=2(\Fb(S)-z)\notin 2S$. Since $S$ is irreducible, we deduce that $z\in S$. Whence, $x=2z\in 2S\subset T$. If  $x\in \N^p\setminus2\N^p$ we assume $x\prec \Fb(T)$, otherwise $x\in T$ and the lemma holds. Since $\Fb(T)-x\notin T$ it follows that $\Fb(T)-x\notin A$ and, taking into account that $\Fb(T)-x\notin 2\N^p$, we deduce that $\Fb(T)-x\prec \Fb(S)$. Thus $x\succ \Fb(S)$ and  $x\in A\subset T$.
\end{proof}

\begin{theorem}
    Every $\CaC$-semigroup is one-fourth of infinitely many pseudo-symmetric $\CaC$-semigroups.
\end{theorem}
\begin{proof}
Let $S$ be a $\CaC$-semigroup. Applying Remark \ref{remarkInfSim} we obtain that there exists infinitely many symmetric $\CaC$-semigroups $T$ such that $S=\frac{T}{2}$. For each $T$, Lemma \ref{lemaSsymTpsdo} guarantees the existence of a pseudo-symmetric $\CaC$-semigroup $T'$ such that $S=\frac{T'}{4}$.
\end{proof}

To sum it up, as a consequence of  Corollary \ref{coroTpsdo-Sirre} and Lemma \ref{lemaSsymTpsdo}, we obtain the announced result.

\begin{theorem}\label{cairrquo}
    A $\CaC$-semigroup is irreducible if and only if it is one-half of a pseudo-symmetric $\CaC$-semigroup.
\end{theorem}

\section{Arithmetic varieties of affine semigroups}\label{av}

Building upon our investigation of affine semigroup quotients, we explore arithmetic varieties defined by such quotients, which provide valuable insights into the structure of semigroups. This section introduces two intriguing families: one arising from systems of modular Diophantine inequalities, where integer solutions modulo a positive integer reveal key information about the underlying semigroup; and the other generalizing the crucial Arf property from numerical semigroups to a broader class of affine semigroups, offering new avenues for analysis (see \cite{libroRosales} for an introduction to the Arf property).

Recall that an arithmetic variety is a non-empty family of affine semigroups $\CaA$ that satisfy:
\begin{itemize}
    \item If $S, T \in \CaA$, then $S \cap T \in \CaA$.
    \item If $S \in \CaA$ and $d \in \N \setminus \{0\}$, then $\frac{S}{d} \in \CaA$.
\end{itemize}

It could be deduced that the intersection of arithmetic varieties is again an arithmetic variety. In \cite[Proposition 4]{RosalesAV}, it is proved that given a family of numerical semigroups, the intersection of all the arithmetic varieties containing this family is an arithmetic variety. Furthermore, it is the smallest that satisfies the previous condition. This also holds for affine semigroups.

A modular Diophantine inequality is an expression of the form $ax \mod b \leq cx$ with $a$, $b$, $c$ integers such that $b\neq 0$. This kind of inequality is introduced in \cite{R-G-U}, and the set of its non-negative solutions is a numerical semigroup, usually named proportionally modular numerical semigroup. If a system of modular Diophantine inequations is considered, the set of its non-negative solutions is also a numerical semigroup.

In \cite{G-M-V18}, the concept of modular Diophantine inequality is generalised to more than one variable: a modular Diophantine inequality with $p$ variables is an expression $f(x)\mod b\leq g(x)$ for some $f,g:\Q^p\to \Q$ two non-null linear functions, and a non-zero natural number $b$. In that case, its non-negative solutions make up an affine semigroup whose associated integer cone is $\N^p$. The same holds if you take any system of modular Diophantine inequations with $p$ variables. A semigroup $S$ is a proportionally modular affine semigroup if and only if there exist $a_1,\ldots, a_p\in \N$, and $g_1,\ldots, g_p\in \Z$, and a natural number $b$ such that $S=\{x\in\N^p\mid \sum_{i=1}^p a_ix_i\mod b\leq \sum_{i=1}^p g_ix_i\}$. These semigroups are not necessarily $\N^p$-semigroups. For example, if one fix $a_1=0$ and $g_1\geq 0$, then $S$ includes the infinite set $\{(x,0)\in\N^p\mid x\in \N\}$. We use the same designation (proportionally modular affine semigroup) when $S$ is defined for more than one modular Diophantine inequality. In that case, such semigroup can be expressed as $S=\{x\in\N^p\mid Ax \mod b\leq G x\}$, where $A$ and $G$ are two integers $(k \times p)$-matrices and $b$ is a positive integer matrix with $k$ entries. We can assume that all the entries in the $i$-th row of $A$ are non-negative integers lesser than $b_i$.

\begin{proposition}
The set of the proportionally modular affine semigroups in $p$ variables is an arithmetic variety.
\end{proposition}

\begin{proof}
Let $S$ and $S'$ be two proportionally modular affine semigroups in $p$ variables. Trivially, $S\cap S'$ is the proportionally modular affine semigroup defined by the union of the modular Diophantine inequalities of $S$ and $S'$.

If $S$ is determined by $\{x\in\N^p\mid Ax \mod b\leq G x\}$, and $d$ is a positive integer, then $\frac{S}{d}$ is also a proportionally modular affine semigroup defined by $\{x\in\N^p\mid dAx \mod b\leq dG x\}$. So, the proposition holds.
\end{proof}

We say that an affine semigroup $S$ has the \textit{Arf property} or is an \textit{Arf} (affine) semigroup if, for any $x,y,z\in S$ with $x\geq y \geq z$, then $x+y-z\in S$. We say that an affine semigroup $S$ is \textit{saturated} if $s,s_1,\ldots, s_r\in S$ are such that $s_i\leq s$ for all $i\in[r]\setminus\{0\}$ and $z_1,\ldots, z_r\in\Z$ are such that $z_1s_1+\cdots+z_rs_r\in \N^p$  then, $s+z_1s_1+\cdots+z_rs_r\in S$. Note that, as occurs for numerical semigroups, a saturated affine semigroup has the Arf property \cite[Lemma 3.31]{libroRosales}.

\begin{proposition}
    The set of Arf affine semigroups is an arithmetic variety.
\end{proposition}
\begin{proof}
    Observe that $S=\N^p$ is an Arf semigroup. On the one hand, for every two semigroups $S$ and $T$ with the Arf property, $S\cap T$ is an Arf semigroup. On the other hand, note for any positive integer $d$ and any Arf semigroup $S$, since $S$ is an Arf semigroup, it follows that $d(x+y-z)\in S$, for any $x,y,z\in \frac{S}{d}$ such that $x\geq y \geq z$. Whence, $\frac{S}{d}$ has the Arf property.
\end{proof}

\begin{proposition}
    The set of saturated affine semigroups is an arithmetic variety.
\end{proposition}
\begin{proof}
    Observe that $S=\N^p$ is a saturated semigroup. As occurs for Arf semigroups, the intersection of saturated semigroups is again a saturated semigroup. Consider a positive integer $d$ and a saturated semigroup $S$. Given $x,x_1,\ldots,x_r\in \frac{S}{d}$ such that $x_i\leq x$, for all, $i\in[r]\setminus\{0\}$ and $z_1,\ldots, z_r\in\Z$ such that $z_1x_1+\cdots+z_rx_r\in \N^p$, $d(s+z_1s_1+\cdots+z_rs_r)\in S$ since $S$ is saturated. Whence, $\frac{S}{d}$ is a saturated semigroup.
\end{proof}

\subsection*{Funding}

The first and third-named authors were supported partially by Junta de Andalucía research group FQM 343.

The last author is partially supported by grant PID2022-138906NB-C21 funded by MCIN/AEI/10.13039/501100011033 and by ERDF ''A way of making Europe''.

Consejería de Universidad, Investigación e Innovación de la Junta de Andalucía project ProyExcel\_00868 also partially supported all the authors.

Proyectos de investigación del Plan Propio – UCA 2022-2023 (PR2022-011 and PR2022-004) partially supported the first, and third-named authors.

This publication and research have been partially granted by INDESS (Research University Institute for Sustainable Social Development), Universidad de Cádiz, Spain.

\subsubsection*{Author information}

J. I. Garc\'{\i}a-Garc\'{\i}a. Departamento de Matem\'aticas/INDESS (Instituto Universitario para el Desarrollo Social Sostenible),
Universidad de C\'adiz, E-11510 Puerto Real (C\'{a}diz, Spain).
E-mail: ignacio.garcia@uca.es.

\noindent
R. Tapia-Ramos. Departamento de Matem\'aticas, Universidad de C\'adiz, E-11406 Jerez de la Frontera (C\'{a}diz, Spain).
E-mail: raquel.tapia@uca.es. 

\noindent
A. Vigneron-Tenorio. Departamento de Matem\'aticas/INDESS (Instituto Universitario para el Desarrollo Social Sostenible), Universidad de C\'adiz, E-11406 Jerez de la Frontera (C\'{a}diz, Spain).
E-mail: alberto.vigneron@uca.es.

\end{document}